\documentclass{amsart}
\usepackage{amssymb}
\newtheorem{theorem}{Theorem}[section] 
\newtheorem{lemma}[theorem]{Lemma}     
\newtheorem{corollary}[theorem]{Corollary}
\newtheorem{proposition}[theorem]{Proposition}
\usepackage{dsfont}

\theoremstyle{definition}

\theoremstyle{remark}
\newtheorem{remark}{Remark}

\newcommand{\I}{{\mathds {1}}}

\begin{document}

\title[Invariant subspaces for $H^2$ spaces]%
{Invariant subspaces for $H^2$ spaces of $\sigma$-finite algebras}

\author{Louis Labuschagne}

\address{DST-NRF CoE in Math. and Stat. Sci,\\ Unit for BMI,\\ Internal Box 209, School of Comp., Stat., $\&$ Math. Sci.\\
NWU, PVT. BAG X6001, 2520 Potchefstroom\\ South Africa}
\email{louis.labuschagne@nwu.ac.za}


\thanks{This work is based on research supported by the National Research Foundation. Any opinion, findings and conclusions or recommendations expressed in this material, are those of the author, and therefore the NRF do not accept any liability in regard thereto.}
\date{\today}

\begin{abstract}
We show that a Beurling type theory of invariant subspaces of noncommutative $H^2$ spaces holds true in the setting of subdiagonal subalgebras of $\sigma$-finite von Neumann algebras. This extends earlier work of Blecher and Labuschagne \cite{BL-Beurling} for finite algebras, and complements more recent contributions in this regard by Bekjan \cite{Bek2} and Chen, Hadwin and Shen \cite{CHS} in the finite setting, and Sager \cite{sager} in the semifinite setting.

We then also introduce the notion of an analytically conditioned algebra, and go on to show that in the class of analytically conditioned algebras this Beurling type theory is part of a list of properties which all turn out to be \emph{equivalent} to the maximal subdiagonality of the given algebra.
\end{abstract}

\maketitle

\section{Background and Introduction}

In the late 50's and early 60's of the previous century, it became apparent that many
famous theorems about the classical $H^\infty$ space of bounded analytic functions on the disk, could be generalized to the setting of abstract function algebras. Many notable researchers contributed to the development of these ideas; in particular
Helson and Lowdenslager \cite{HL}, and Hoffman \cite{Ho}. The paper
\cite{SW} of Srinivasan and Wang, from the middle of the 1960s, organized and summarized much of this
`commutative generalized $H^p$-theory'. The construct that Srinivasan and Wang used to summarise these results in \cite{SW}, was the so-called {\em weak* Dirichlet algebras}. Essentially this summary furnishes one with an array of properties that are all in some way equivalent to the validity of a Szeg\"o formula for these weak* Dirichlet algebras.

Round about the same time that the paper of Srinivasan and Wang appeared, Arveson introduced his notion of subdiagonal subalgebras of von Neumann algebras as a possible context for extending this cycle of results to the noncommutative context \cite{PTAG,AIOA}. In the case that $A$ is a maximal subdiagonal subalgebra of a von Neumann algebra $M$ equipped with a finite trace (all concepts defined below), $H^p$ may be defined to be the closure of $A$ in the noncommutative $L^p$ space
$L^p(M)$. In the case where $A$ contains no selfadjoint elements except scalar multiples of the identity, 
the $H^p$ theory will in the setting where $M$ is commutative, collapse to the classical theory of $H^p$-spaces associated to weak* Dirichlet algebras. Thus Arveson's setting canonically extends the notion of weak*
Dirichlet algebras.

The theory of these subdiagonal algebras progressed at a carefully measured pace, until in 2005, Labuschagne \cite{L-Szego} managed to use some of Arveson's ideas to show that in the context of finite von Neumann algebras, these maximal subdiagonal algebras satisfy a 
Szeg\"o formula. 

Pursuant to this breakthrough, in a sequence of papers (\cite{BL1}, \cite{BL-FMRiesz}, \cite{BL-outer}, \cite{BL-Beurling}, 
\cite{BL-outer2}), complemented by important contributions from Ueda \cite{Ueda}, and Bekjan and Xu \cite{BX}, Blecher and Labuschagne demonstrated that in the context of finite von Neumann algebras the \emph{entire} cycle of results (somewhat surprisingly) survives the passage to noncommutativity. Specifically it was shown that the same cycle of results hold true for what Blecher and Labuschagne call tracial subalgebras of a finite von Neumann algebra (see \cite{BLsurvey}). 

With the theory of subdiagonal subalgebras of finite von Neumann algebras thereby reaching some level of maturity, several authors then turned their attention to the the analysis of the case of $\sigma$-finite von Neumann algebras. Important structural results were obtained by Ji, Ohwada, Saito, Bekjan and Xu (\cite{JOS}, \cite{jisa}, \cite{Xu}, \cite{jig2}, \cite{jig1}, \cite{Bek}). 

However the transition from finite to $\sigma$-finite von Neumann algebras cannot be made without some sacrifice. One very costly price paid for the passage to the $\sigma$-finite case, is the loss of the theory of the Fuglede-Kadison determinant (\cite{FKa}, \cite{AIOA}). (As was shown by Sten Kaijser \cite{Kai}, the presence of such a determinant forces the existence of a finite trace, and hence the theory of the Fuglede-Kadison determinant, is is essentially a theory of finite von Neumann algebras.) In the case of subdiagonal subalgebras of finite von Neumann algebras, this determinant played the role of a geometric mean, and hence featured very prominently in the development of that theory. But with no such determinant, how does one even begin to give a sensible and useful description of a geometric mean, and with no geometric mean, how can one give expression to a Szeg\"o formula in this context? 

In this paper we show that despite this very formidable challenge, there are nevertheless several aspects of the tracial theory which survives the transition to the type III case. These aspects include a very detailed Beurling-type theory of invariant subspaces, and an extension of the so-called \emph{unique normal state extension property}. (One version of the unique normal state extension property amounts to the claim that any $f\in L^1(M)^+$ will belong to $L^1(\mathcal{D})$ whenever $f\perp (A\cap \mathrm{ker}(\mathcal{E}))$, where $\mathcal{E}$ is a conditional expectation from $M$ onto $A\cap A^*$.) In fact these theories not only hold for type III maximal subdiagonal algebras, but serve to characterise them among the class of what we will call \emph{analytically conditioned} subalgebras (definition loc. cit.). See Theorem \ref{Co}. As we shall see, in many cases the proofs turn out to be remarkably similar to those of Blecher and Labuschagne in \cite{BL-outer}, with important and at times quite subtle technical modifications needing to be made at crucial points.

Throughout $M$ will be a $\sigma$-finite von Neumann algebra equipped with a faithful normal state $\varphi$. A weak*-closed unital subalgebra $A$ of $M$ will be called subdiagonal, if there exists a faithful normal conditional expectation $\mathcal{E}$ onto the subalgebra $\mathcal{D}=A\cap A^*$, which is also multiplicative on $A$. Here $\mathcal{D}=A\cap A^*$ is sometimes referred to as the diagonal. In cases where the identity of the diagonal is important, we will say that $A$ is subdiagonal with respect to $\mathcal{D}$. We pause to point out that the assumption regarding the weak*-closedness of $A$ does not generally form part of the definition of subdiagonality. But since we are primarily interested in studying maximal subdiagonal algebras, and since the weak* closure of an algebra that is subdiagonal with respect to $\mathcal{D}$ will also be subdiagonal with respect to $\mathcal{D}$, we may make this assumption without any loss of generality. 

The following theorem characterises those subdiagonal algebras which are maximal with respect to a given diagonal $\mathcal{D}$. We pause to give some insight into this theorem. With $A$ a subdiagonal algebra and $\mathcal{D}$ and $\mathcal{E}$ as above, the condition $\varphi\circ\mathcal{E}=\varphi$ turns out to be equivalent to the claim that $\sigma_t^\varphi(\mathcal{D})=\mathcal{D}$ for all $t\in \mathbb{R}$. In fact the very existence of $\mathcal{E}$ is ensured by the fact that the maps $\sigma_t^\varphi$ ``preserve'' $\mathcal{D}$. (See \cite[Theorem IX.4.2]{Tak}.) However if alternatively we had that the maps $\sigma_t^\varphi$ preserve $A$, the fact that they would then also preserve $\mathcal{D}$, is a trivial consequence of the fact that $\mathcal{D}=A\cap A^*$. Hence such preservation of $A$ by these maps, is more restrictive than preservation of $\mathcal{D}$, and as such guarantees the existence of $\mathcal{E}$. As can be seen from the theorem, if $A$ is large enough to ensure that $A+A^*$ is weak*-dense in $M$, then maximality with respect to $\mathcal{D}$ is signified by precisely this more restrictive requirement.

\begin{theorem}[\cite{Xu}, \cite{JOS}] Let $A$ be a weak* closed unital subalgebra of $M$ with $\mathcal{D}$ and $\mathcal{E}$ as before, and assume that additionally $A+A^*$ is weak*-dense in $M$. Then $A$ is maximal as a subdiagonal subalgebra with respect to $\mathcal{D}$ if and only if  $\sigma_t^\varphi(A)=A$ for all $t\in \mathbb{R}$.
\end{theorem}

\begin{proof} The ``if'' part follows from \cite[Theorem 1.1]{Xu}. The ``only if'' part from \cite[Theorem 2.4]{JOS}. 
\end{proof}
 
Let $A$, $\mathcal{D}$ and $\mathcal{E}$ be as before. The above result may alternatively be interpreted as the statement that any weak*-closed subdiagonal subalgebra $A$ for which we have that $\sigma_t^\varphi(A)=A$ for all $t\in \mathbb{R}$, will be maximal subdiagonal with respect to $\mathcal{D}$ whenever $A+A^*$ is weak*-dense in $M$. It is this interpretation that we use as our starting point. We will therefore call any weak* closed unital subalgebra $A$ of $M$ for which 
\begin{itemize}
\item $\sigma_t^\varphi(A)=A$ for all $t\in \mathbb{R}$,
\item and for which the faithful normal conditional expectation $\mathcal{E}:M\to \mathcal{D}=A\cap A^*$ satisfying $\varphi\circ\mathcal{E}=\varphi$ (ensured by the above condition \cite[Theorem IX.4.2]{Tak}), is multiplicative on $A$
\end{itemize}
an \emph{analytically conditioned} subalgebra. In the case that $\varphi$ is actually a tracial state, Blecher and Labuschagne called such algebras \emph{tracial} algebras. If additionally we assume that $A+A^*$ is $\sigma$-weakly dense in $M$, then $(M, A, \mathcal{E}, \varphi)$ is what Prunaru calls a subdiagonal quadruple \cite{pru}. However in deference to the preceding theorem and following GuoXing Ji, we will simply refer to such algebras as maximal subdiagonal. Given an analytically conditioned algebra, our objective in this paper is then to look for properties that may be compared to the criterion of requiring $A+A^*$ to be weak*-dense in $M$.

For the sake of simplicity we will in the discussion that follows write $L$ for $M\rtimes_\varphi \mathbb{R}$. The crossed product of course admits a dual action of $\mathbb{R}$ in the form of an automorphism group $\theta_s$ and a canonical trace characterised by the property that $\tau_L\circ \theta_s=e^{-s}\tau_L$. The $L^p$-spaces are defined as $L^p(M)=\{a\in \widetilde{L}: \theta_s(a)=e^{-s/p}a \mbox{ for all } s\in \mathbb{R}\}$. The space $L^1(M)$ admits a canonical trace functional $tr$, which is used to define a norm $\|a\|=tr(|a|^p)^{1/p}$ on $L^p(M)$. The topology on $L^p(M)$ engendered by this norm, coincides with the relative topology of convergence in measure that $L^p(M)$ inherits from $\widetilde{L}$.

Now let $h=\frac{d\widetilde{\varphi}}{d\tau_L}\in L^1(M)$. It is well-known that $L^p(M)$ may for any $0\leq c\leq 1$ be realised as the completion of $\{h^{c/p}fh^{(1-c)/p}: f\in M\}$ under the norm $tr(|a|^p)^{1/p}$ \cite{Kos}. Given $1\leq p <\infty$, we know from the work of Ji \cite[Theorem 2.1]{jig1} that for any maximal subdiagonal algebra $A$, the closures of each of $\{h^{c/p}fh^{(1-c)/p}: f\in A\}$ in $L^p(M)$ (where $0\leq c\leq 1$), all agree. It is this closure that we will identify as our Hardy spaces $H^p(A)$. However a careful perusal of \cite[Theorem 2.1]{jig1}, reveals that all we need for the proof of that theorem to go through, is the invariance of $A$ under the action of $\sigma_t^\varphi$. Hence even for analytically conditioned algebras we have that the closures of $\{h^{c/p}fh^{(1-c)/p}: f\in A\}$ agree for each $0\leq c\leq 1$. Note that this fact ensures that these closures are all right $\mathcal{D}$-modules. In the case where we are dealing with analytically conditioned algebras, we will write $\mathcal{H}^p(A)$ for these closures, and occasionally refer to this subspace of $L^p(M)$ as the $L^p$-hull of $A$.  For subspaces $X$ of $L^p(M)$, we will simply write $[X]_p$ for the closure in $L^p(M)$. Ji also showed that for maximal subdiagonal algebras, $L^p(M)=H^p(A)\oplus H^p_0(A)^*$ for any $1<p<\infty$ \cite[Theorem 3.3]{jig1}. 

We recall that a {\em (right) invariant subspace} of $L^p(M)$, is
a closed subspace $K$ of $L^p(M)$ such that $K A \subset K$.
For consistency,
we will not consider left invariant subspaces at all, leaving the reader to verify
that entirely symmetric results pertain in the left invariant case.
An invariant subspace is called {\em simply invariant} if in addition
the closure of $K A_0$ is properly contained in $K$.

If $K$ is a right $A$-invariant subspace of $L^2(M)$, we define
the {\em right wandering subspace} of $K$ to
be the space $W = K \ominus [K A_0]_2$;
and we say that $K$ is {\em type 1} if
$W$ generates $K$ as an $A$-module (that is, $K = [W A]_2$).
We will say that $K$ is
{\em type 2} if $W = (0)$. 

\section{Invariant subspaces and the module action of $\mathcal{D}$}

We pause to review some necessary technical facts regarding faithful normal conditional expectations, before proceeding with the analysis.

\begin{remark}\label{expect} We proceed to review some basic properties of the expectation $\mathcal{E}$ in this context. The basic references we will use for properties of expectations are \cite{Gol} and \cite{Junge}. It is instructive to note that $\mathcal{D}\rtimes_{\sigma^{\varphi}}{\mathbb R}$, can be realised as a subalgebra of $L=M\rtimes_{\sigma^{\varphi}}{\mathbb R}$. In fact $\mathcal{E}$ extends canonically to a conditional expectation from $M\rtimes_{\sigma^{\varphi}}{\mathbb R}$ to $\mathcal{D}\rtimes_{\sigma^{\varphi}}{\mathbb R}$, which we will here denote by $\overline{\mathcal{E}}$. Moreover for any $1\leq p<\infty$ this extension canonically induces an expectation $\mathcal{E}_p$ from $L^p(M)$ to $L^p(\mathcal{D})$. Note in particular that

\begin{itemize}
\item $\overline{\mathcal{E}}\circ\theta_s=\theta_s\overline{\mathcal{E}}\circ$ for any $s$ where denotes the dual action of $\theta_s$ ${\mathbb R}$ on $M\rtimes_{\sigma^{\varphi}}{\mathbb R}$. \cite[4.4]{Gol}.
\item With $\widetilde{\varphi}$ denoting the dual weight on $L=M\rtimes_{\sigma^{\varphi}}{\mathbb R}$ and $\tau_L$ the canonical trace on the crossed product, $\overline{\mathcal{E}}$ is both $\widetilde{\varphi}$ and $\tau_L$ invariant. \cite[Theorem 4.7]{Gol}
\item $\mathcal{E}_1$ maps $h_M=\frac{d \widetilde{\varphi}}{d \tau_L}$ onto $h_{\mathcal{D}}=\frac{d \widetilde{\varphi}}{d \tau_{\mathcal{D}\rtimes{\mathbb R}}}$. \cite[Lemma 4.8]{Gol}, \cite[Lemma 2.1]{Junge}.
\item $\overline{\mathcal{E}}$ extends canonically to the extended positive part of $M\rtimes_{\sigma^{\varphi}}{\mathbb R}$. When restricted to $L^p_+(M)$ ($1\leq p<\infty$), this extension coincides with the restriction of $\mathcal{E}_p$.
\item For $s\geq 1$, $\frac{1}{s}=\frac{1}{p}+\frac{1}{q}+\frac{1}{r}$, $a\in L^p(\mathcal{D})$, $b\in L^p(M)$ and $c\in L^p(\mathcal{D})$, we have $\mathcal{E}_s(abc)=a\mathcal{E}_q(b)c$. \cite[2.5]{Junge}
\item For any $a\in L^1(M)$, we have that $tr(\mathcal{E}_1(a))=tr(a)$ where $tr$ is the canonical trace functional on $L^1(M)$. See \cite[Lemma 2.1]{Junge} and the discussion preceding \cite[2.5]{Junge} where it is noted that $\mathcal{E}_1=\overline{\mathcal{E}}_*$.
\end{itemize}
\end{remark}

\begin{proposition} Let $A$ be an analytically conditioned algebra. Given $r\geq 1$ with $\frac{1}{r}=\frac{1}{p}+\frac{1}{q}$, and given $a\in \mathcal{H}^p(A)$ and $b\in \mathcal{H}^q(A)$, we have that $ab\in \mathcal{H}^r(A)$ with $\mathcal{E}_r(ab)=\mathcal{E}_p(a)\mathcal{E}_q(b)$.
\end{proposition}

\begin{proof} Given $a_0, b_0\in A$, we have that 
\begin{eqnarray*}
\mathcal{E}_r((h^{1/p}a_0)(b_0h^{1/q}))&=&h^{1/p}\mathcal{E}(a_0b_0)h^{1/q}\quad\mbox{\cite[2.5]{Junge}}\\
&=&(h^{1/p}\mathcal{E}(a_0))(\mathcal{E}(b_0)h^{1/q})\quad\mathcal{E}\mbox{ is multiplicative on }A\\
&=&\mathcal{E}_p(h^{1/p}a_0)\mathcal{E}_q(b_0h^{1/q})\quad\mbox{\cite[2.5]{Junge}}
\end{eqnarray*}
The result follows on extending the actions of $\mathcal{E}_p$, $\mathcal{E}_q$ and $\mathcal{E}_r$ by continuity.\end{proof}

In the following we will where there is no danger of confusion, drop the subscript $p$ when denoting the action of $\mathcal{E}$ on $L^p(M)$.

\begin{corollary}\label{perp} For any analytically conditioned algebra $A$, we have that 
$\mathcal{H}^2(A)+(\mathcal{H}^2(A)^*)= \mathcal{H}^2(A)\oplus L^2(\mathcal{D})\oplus\mathcal{H}^2_0(A)^*$ where $\mathcal{H}^2_0(A)=\{f\in \mathcal{H}^2(A):\widehat{\mathcal{E}}(f)=0\}$.
\end{corollary}

\begin{proof} Given any $f\in \mathcal{H}^2_0(A)$ and $g\in \mathcal{H}^2(A)$, it is a simple matter to see that $\langle g, f^*\rangle=tr(fg)=tr\circ\mathcal{E}(fg) = tr(\mathcal{E}(f)\mathcal{E}(g))=0$. Hence $\mathcal{H}^2_0(A)^*\perp \mathcal{H}^2(A)$. In particular the subspace $L^2(\mathcal{D})$ of $\mathcal{H}^2(A)\cap \mathcal{H}^2(A)$ is also orthogonal to $\mathcal{H}^2_0(A)^*$. Since for any $g\in \mathcal{H}^2(A)$ we have that $\mathcal{E}(f)\in \mathcal{H}^2(M)$ with $\mathcal{E}(f-\mathcal{E}(f))=0$, it follows that $L^2(\mathcal{D})\oplus\mathcal{H}^2_0(A)^*$ is all of $\mathcal{H}^2(A)^*$.
\end{proof}

Using the properties of $\mathcal{E}$ described in the preceding Remark and Proposition, \cite[Theorem 2.1]{BL-Beurling} may now be extended to the $\sigma$-finite setting. The proofs for the two cases are virtually identical, with the primary change needing to be made in the passage from the finite to the $\sigma$-finite case, being that we need to substitute the tracial functional $tr_M$ for the finite trace $\tau_M$ at suitable points. We therefore choose to leave the translation of this proof to the $\sigma$-finite setting as an exercise.

\begin{theorem} \label{inv1}
Let $A$ be an analytically conditioned algebra.    \begin{itemize}
\item [(1)]   Suppose that $X$ is a subspace of $L^2(M)$
of the form $X = Z \oplus^{col} [YA]_2$ where $Z, Y$ are closed
subspaces of $X$, with $Z$
a type 2 invariant subspace,
and $\{y^*x : y, x \in Y \} = Y^*Y \subset L^1({\mathcal D})$.  Then
$X$ is simply right $A$-invariant if and only if $Y \neq
\{0\}$.
\item [(2)]  If $X$ is as in {\rm (1)},
then $[Y {\mathcal D}]_2 = X \ominus [XA_0]_2$ (and
$X = [XA_0]_2 \oplus [Y {\mathcal D}]_2$).
\item [(3)]   If $X$ is as described in {\rm (1)},
then that description also holds if $Y$ is replaced by $[Y {\mathcal D}]_2$.  Thus
(after making this replacement)
we may assume that $Y$ is a ${\mathcal D}$-submodule of $X$.

\item [(4)]   The subspaces $[Y {\mathcal D}]_2$ and $Z$ in the decomposition
in  {\rm (1)} are uniquely determined by $X$.  So is $Y$ if we
take it to be a ${\mathcal D}$-submodule (see {\rm (3)}).
\item [(5)]  If $A$ is maximal subdiagonal, then any right $A$-invariant subspace
$X$ of
$L^2(M)$ is of the form described in {\rm (1)},
with $Y$ the right wandering subspace of $X$.
\end{itemize}
\end{theorem}

Building on Theorem \ref{inv1}, we are now able to present the following rather elegant decomposition of the right wandering subspace. This extends \cite[Proposition 2.2]{BL-Beurling}. Although there are close similarities between the proofs of the tracial and the $\sigma$-finite case, there are rather delicate modifications that need to be made for the proof to go through in the general case -- a mere notational change will not suffice. 

\begin{proposition} \label{newpr}  Suppose that $X$ is
as in Theorem {\rm
\ref{inv1}}, and that $W$ is the right wandering subspace of $X$.
Then $W$ may be decomposed as
an orthogonal direct sum $\oplus^2_i \, u_i L^2({\mathcal D})$,
where $u_i$ are partial isometries in $M$ for which 
$u_i(\frac{d\widetilde{\varphi}}{d\tau_L})^{1/2}\in W$, with
$u_i^* u_i \in {\mathcal D}$, and $u_j^* u_i = 0$ if
$i \neq j$.   If $W$ has a cyclic vector for the ${\mathcal D}$-action,
then we need only one partial isometry in the above.
\end{proposition}

\begin{proof}   By the theory of representations of a
von Neumann algebra (see e.g.\ the discussion at the
start of Section 3 in
\cite{js}), any normal Hilbert
${\mathcal D}$-module is an $L^2$ direct sum of cyclic
Hilbert ${\mathcal D}$-modules,
and if $K$ is a normal
cyclic Hilbert ${\mathcal D}$-module, then
$K$ is spatially isomorphic to $eL^2({\mathcal D})$, for
an orthogonal projection $e \in {\mathcal D}$.

Suppose that the latter isomorphism is implemented by a unitary
${\mathcal D}$-module map $\psi$. If in addition $K \subset W$, let $g = \psi(eh^{1/2}) \in W$ where $h=\frac{d\widetilde{\varphi}}{d\tau_L}$.  Then $tr(d^* g^* g d) =
\Vert \psi(e d) \Vert_2^2 = tr(d^* h^{1/2}eh^{1/2}d)$, for each $d \in {\mathcal D}$.
By Theorem \ref{inv1}, $u^* u \in L^1({\mathcal D})$,  and so $g^* g = h^{1/2}eh^{1/2}$. 
Hence there exists a partial isometry $u$ with initial projection $e$ such that $g=ueh^{1/2}=uh^{1/2}$. 
the modular action of $\psi$ we will then have that $\psi(eh^{1/2}d)=\psi(eh^{1/2})d=uh^{1/2}d$ 
for any $d\in \mathcal{D}$. Since $L^2(\mathcal{D})$ is the closure of $\{h^{1/2}d: d\in \mathcal{D}\}$, it follows 
that $\psi(eL^2(\mathcal{D}))=uL^2(\mathcal{D})$.
 
Given $u_i$ and $u_j$ with $i\neq j$, we have that $u_iL^2(\mathcal{D}), u_jL^2(\mathcal{D})\subset W$. 
Hence $L^2(\mathcal{D})u_j^* u_iL^2(\mathcal{D}) \subset L^1({\mathcal D})$. Since for 
any $d\in \mathcal{D}$ we have that  $tr(h^{1/2}u_j^* u_ih^{1/2} d) = tr(\psi(e_jh^{1/2})^*\psi(e_ih^{1/2})d)= tr(\psi(e_jh^{1/2})^*\psi(e_ih^{1/2}d))=tr(h^{1/2}e_je_ih^{1/2}d)=0$, it follows from the previously mentioned fact that 
$h^{1/2}u_j^* u_ih^{1/2}=0$, and hence that $u_j^* u_i=0$. (To see this recall that the embedding $M\to L^2(M):a\to h^{1/2}eh^{1/2}$ is injective \cite{Kos}.) In the case where $i=j$ we of course have that 
$u_i^*u_i=e_i\in \mathcal{D}$. Putting these facts together,
we see that $W$ is of the desired form.
\end{proof}

\begin{corollary}  Suppose that $X$ is
as in Theorem {\rm \ref{inv1}}, and that $W$ is the right wandering subspace of $X$.
If indeed $X\subset \mathcal{H}^2(A)$, then $Z\perp L^2(\mathcal{D})$. If additionally $A$ is maximal subdiagonal, then the partial isometries $u_i$ described in the preceding Proposition, all belong to $A$.
\end{corollary}

\begin{proof} If indeed $X\subset\mathcal{H}^2(A)$, it is a fairly trivial observation to make that $Z=[ZA_0]_2 \subset [XA_0]_2\subset [\mathcal{H}^2(A)A_0]_2=\mathcal{H}^2_0(A)$. It is clear from the proof of Corollary \ref{perp} that $\mathcal{H}^2(A)=\mathcal{H}^2_0(A)\oplus L^2(\mathcal{D})$, and hence the first claim follows. 

Now suppose that $A$ is maximal subdiagonal. To see the second claim recall that in the proof of Proposition \ref{newpr}, we showed that $u_iL^2(\mathcal{D})\subset W$ for each $i$. Hence given any $a\in A_0$, and taking $h=\frac{d\widetilde{\varphi}}{d\tau_L}$, we will therefore have that $au_ih^{1/2}\in aW \subset A_0X \subset A_0H^2(A) \subset H^2_0(A)$. But $\mathcal{E}_2$ annihilates $H^2_0(A)$, and hence we must have that $0=\mathcal{E}_2(au_ih^{1/2})=\mathcal{E}(au_i)h^{1/2}$. It now follows from the injectivity of the injection $M\to L^2(M):f\to fh^{1/2}$ (see \cite{Kos}), that $\mathcal{E}(au_i)=0$. Since $a\in A_0$ was arbitrary, we may now apply \cite[Theorem 2.2]{jisa} to conclude that $u_i\in A$ as claimed. 
\end{proof}

\begin{corollary} \label{adcor}  If $X$ is an
invariant subspace of the form described in
Theorem {\rm \ref{inv1}}, then $X$ is type 1 if and only if
 $X = \oplus^{col}_i \, u_i \mathcal{H}^2(A)$, for $u_i$ as
in Proposition {\rm \ref{newpr}}.
\end{corollary}

\begin{proof}  If $X$ is type 1, then $X = [W A]_2$
where $W$ is the right wandering space, and so the one assertion
follows from Proposition {\rm \ref{newpr}}.
If $X = \oplus^{col}_i \, u_i \mathcal{H}^2(A)$, for $u_i$ as above, then
$[X A_0]_2 = \oplus^{col}_i \, u_i \mathcal{H}^2(A_0)$, and from this
it is easy to argue that  $W =
\oplus^{col}_i \, u_i L^2({\mathcal D})$.
Thus $X = [W A]_2 = \oplus^{col}_i \, u_i \mathcal{H}^2(A)$.
  \end{proof}
  
The following Theorem extends \cite[Proposition 2.4]{BL-Beurling}. Although the proofs of the two cases are almost identical, 
there was a typo in (ii) and (iv) of \cite[Proposition 2.4]{BL-Beurling}. (The column sum $K_1\oplus^{col}K_2$ should've been 
$K_2\oplus^{col}K_1$.) For this reason we choose state the proof in full. 

\begin{proposition} \label{typestuff}   Let $X$ be a
closed $A$-invariant subspace of $L^2(M)$, where $A$ is an
analytically conditioned subalgebra of $M$.
\begin{itemize}  \item [(1)]
If $X = Z \oplus [Y A]_2$ as in Theorem
 {\rm \ref{inv1}}, then $Z$ is type 2, and $[Y A]_2$ is type 1.
 \item [(2)]  If $A$ is a maximal subdiagonal algebra,
and if $X = K_2 \oplus^{col} K_{1}$ where
$K_1$ and $K_{2}$ are types 1 and 2 respectively,
then $K_1$ and $K_2$ are respectively the unique
spaces $Z$ and $[Y A]_2$ in  Theorem  {\rm \ref{inv1}}.
 \item [(3)]  If $A$ and $X$ are as in {\rm (2)},
and if $X$ is type 1 (resp.\ type 2),
then the space  $Z$ of Theorem  {\rm \ref{inv1}}
for $X$ is $(0)$ (resp.\ $Z = X$).
 \item [(4)]   If  $X = K_2 \oplus^{col} K_{1}$ where
$K_1$ and $K_{2}$ are types 1 and 2 respectively,
then the right wandering subspace for $X$
equals the right wandering subspace for $K_1$.
  \end{itemize} \end{proposition}

\begin{proof}  (1) \  Clearly in this case
$Z$ is type 2.  To see that $[Y A]_2$ is type 1, note that
since $Y \perp X A_0$ by part (ii) of Theorem 
{\rm \ref{inv1}}, we must have $Y \perp Y A_0$.
Thus $Y \subset [Y A]_2 \ominus [Y A_0]_2$,
and consequently $[Y A]_2 = [([Y A]_2 \ominus [Y A_0]_2) A]_2$.

(2) \  Suppose that $X = K_2 \oplus^{col} K_{1}$ where
$K_1$ and $K_{2}$ are types 1 and 2 respectively.
Let $Y$ be the right wandering space for $K_1$. Then of course 
$K_1 = [Y A]_2$. By Theorem \ref{inv1} we have $Y^* Y \subset L^1({\mathcal D})$.
So $X =  K_{2} \oplus^{col} [Y A]_2$,
and by the uniqueness assertion in Theorem \ref{inv1},
$K_2$ is the space $Z$ in Theorem \ref{inv1} for $X$.

(3) \ This is obvious from Theorem {\rm \ref{inv1}}.

(4) \ If $K = K_2 \oplus^{col} K_{1}$ as above,
then $K_{2} = [K_{2} A_0]_2 \subset [K A_0]_2$,
and so $K \ominus [K A_0]_2 \subset K \ominus K_{2} = K_1$.
Thus $K \ominus [K A_0]_2 \subset K_1 \ominus [K_1 A_0]_2$.
Conversely, if $\eta \in K_1 \ominus [K_1 A_0]_2$,
then $\eta  \perp K A_0$ since $\eta\in K_1$ ensures that $\eta^* K_2 = (0)$.
So $\eta \in K \ominus [K A_0]_2$.
\end{proof}

On collecting the information reflected in the preceding four results, we obtain the following structure theorem for invariant subspaces.

\begin{theorem} \label{main}  If $A$ is a maximal subdiagonal
subalgebra of $M$, and if $K$ is a closed right  $A$-invariant subspace 
of $L^2(M)$,
then: \begin{itemize}
\item [(1)]   $K$ may be written uniquely as
an (internal) $L^2$-column sum $K_2 \oplus^{col} K_{1}$
of a type 1 and a type 2 invariant subspace of $L^2(M)$, respectively.
  \item [(2)]  If
$K  \neq (0)$ then $K$ is type 1 if and only if
$K = \oplus_i^{col} \, u_i \, H^2$, for $u_i$  partial isometries
with mutually orthogonal ranges and $|u_i| \in {\mathcal D}$.
\item [(3)]
The right wandering subspace $W$ of $K$
is an $L^2({\mathcal D})$-module in the sense of Junge and Sherman, and in
particular $W^* W \subset L^{1}({\mathcal D})$.
\end{itemize}
\end{theorem}

\section{Characterisations of maximal subdiagonal subalgebras}
 
In order to prove our main theorem, we need to invoke the Haagerup reduction theorem (see \cite{HJX}). The use of the reduction theorem in studying $\sigma$-finite subdiagonal subalgebras, was pioneered by Xu \cite{Xu} in his innovative application of the theorem in studying maximality properties of such algebras. We pause to briefly review the main points of that construction. (Further details may be found in \cite{Xu}, \cite{pru}, \cite{jig2}, \cite{jig1}.)

Let $\mathbb{Q}_D$ be the diadic rationals and let $R=M\rtimes_{\sigma^{\varphi}}\mathbb{Q}_D$. Since $\mathbb{Q}_D$ is discrete, there exists a canonical expectation $\Phi$ from $R$ onto $M$. The dual weight $\widehat{\varphi}$ on $R$ turns out to be a faithful normal state. The Haagerup reduction theorem then informs us that there exists an increasing net $R_n$ of finite von Neumann algebras each equipped with a faithful state $\widehat{\varphi}_n=\widehat{\varphi}|_{R_n}$, and a concomitant net of expectations $\Phi_n:R\to R_n$ for which $\Phi_n\circ \Phi_m=\Phi_m\circ\Phi_n=\Phi_n$ when $n\geq m$. (In the case that $\varphi$ is a state, these nets are in fact a sequences.) Moreover $\cup_{n} R_n$ is $\sigma$-strongly dense in $R$. As far as $L^p$ spaces are concerned, the theorem further tells us that for each $0<p<\infty$, $\cup_n L^p(R_n)$ is dense in $L^p(R)$ with each $L^p(R_n)$ canonically isometric to $L^p(R_n, \tau_n)$, where $\tau_n$ is a canonical normal tracial state on $R_n$. 

For weak*-closed unital maximal subdiagonal subalgebras $A$ of the type described above, the work of Xu tells us that in the case presently under consideration (the case where $\varphi$ is a state), both $A$ and the expectation $\mathcal{E}:M\to \mathcal{D}$ extend to $R$ in such a way that $\widehat{A}$ is a maximal subdiagonal subalgebra of $R$, with the extension $\widehat{\mathcal{E}}$ of $\mathcal{E}$ mapping onto $\widehat{A}\cap \widehat{A}^*=\mathcal{D}\rtimes_{\sigma^{\varphi}}\mathbb{Q}_D$. In fact there is a net of subalgebras $\widehat{A}_n\subset R_n$ such that each $\widehat{A}_n$ is subdiagonal in $R_n$ with respect to \emph{both} $\widetilde{\varphi}_n$ and $\tau_n$, with in addition $\cup_{n=1}^\infty \widehat{A}_n$ $\sigma$-weakly dense in $\widehat{A}$. Here $\widehat{A}$ is just the $\sigma$-weak closure of the span of $\{\lambda(t)\pi(x): t\in \mathbb{Q}_D\}$ and may hence be regarded as representing something like $A\rtimes_{\sigma^{\varphi}}\mathbb{Q}_D$. (Here $\pi$ denotes the canonical $*$-homomorphism embedding $M$ into $R=M\rtimes_{\sigma^{\varphi}}\mathbb{Q}_D$.) We then also have that $\Phi(\widehat{A})=A$. The algebra $\widehat{A}_n$ is just $\widehat{A}_n=\widehat{A}\cap R_n$.

\begin{lemma}
 Let $A$ be an analytically conditioned algebra. Then on applying the same construction outlined above to $A$, $\widehat{A}$ will then be an analytically conditioned subalgebra of $R$, and each $\widehat{A}_n=\widehat{A}\cap R_n$ an analytically conditioned subalgebra of $R_n$.
\end{lemma}

\begin{proof} The latter part of the proof of \cite[Lemma 3.1]{Xu}, where it is shown that in the case where $A$ is maximal subdiagonal $\widehat{\mathcal{E}}$ is multiplicative on $\widehat{A}$ and $\widehat{A}\cap \widehat{A}^*=\mathcal{D}\rtimes_{\sigma^{\varphi}}\mathbb{Q}_D$, carries over verbatim to the present context. Hence the claim regarding $\widehat{A}$ follows. Similarly on removing the sections of the proof of \cite[Lemma 3.2]{Xu} devoted to showing that the $\sigma$-weak density of $\widehat{A}+\widehat{A}^*$ in $R$ ensures the $\sigma$-weak density of $\widehat{A}_n+\widehat{A}_n^*$ in $R_n$, the rest of the proof of this lemma essentially proves that $\widehat{A}_n$ is an analytically conditioned subalgebra of $R_n$. 
\end{proof}

\begin{lemma}\label{L2}
 Let $A$ be an analytically conditioned algebra.  If $L^2(M)=\mathcal{H}^2(A)\oplus(\mathcal{H}_0^2(A))^*$, then also 
$L^2(R)=\mathcal{H}^2(\widehat{A})\oplus(\mathcal{H}_0^2(\widehat{A}))^*$, and $L^2(R_n)=\mathcal{H}^2(\widehat{A}_n)\oplus(\mathcal{H}_0^2(\widehat{A}_n))^*$ for each $n$.
\end{lemma}

\begin{proof}
Let $h_M$ be the density $h_M=\frac{d\widetilde{\varphi}}{d\tau_L}\in L^1(M)$ where $L= M\rtimes_{\sigma^{\varphi}}\mathbb{R}$. Since $A$ is an analytically conditioned algebra, we have that $\mathcal{H}^2(A)=\overline{\{h_M^{1/2}f:f\in A\}}=\overline{\{fh_M^{1/2}f:f\in A\}}$. Given any $x\in M$, the fact that $h_M^{1/2}x\in L^2(M)=\mathcal{H}^2(A)\oplus(\mathcal{H}_0^2(A))^*$, ensures that we may find sequences $\{a_n\},\{b_n\} \subset A$ such that $h_M^{1/2}(a_n + b_n^*)\to h_M^{1/2}x$ in norm in $L^2(M)$.

We may now apply the conclusions of Remark \ref{expect} to the pair $(M,R)$ and the expectation $\Phi: R\to M$, rather than to the pair $(\mathcal{D},R)$ and the expectation $\mathcal{E}: M\to \mathcal{D}$. Hence for each $1\leq p\leq \infty$, $L^p(M)$ may be regarded as a subspace of $L^p(R)$, and under this identification, the density $h_R=\frac{d\widetilde{\widehat{\varphi}}}{d\tau}\in L^1(R)$ may be identified with $h_M$. So with this identification, we have that $h_R^{1/2}(a_n + b_n^*)\to h_R^{1/2}x$ in norm in $L^2(R)$. Now for any $t\in \mathbb{Q}_D$, we may apply the noncommutative H\"older inequality to conclude that $h_R^{1/2}(a_n + b_n^*)\lambda(t)\to h_R^{1/2}x\lambda(t)$ in norm in $L^2(R)$. It is a trivial observation to make that $\{a_n\lambda(t)\}, \{\lambda(t^{-1})b_n\} \}\subset \widehat{A}$, and hence that  $\{(h_R^{1/2}a_n) + (\lambda(t^{-1})b_n)^*)\}=\{h_R^{1/2}(a_n + b_n^*)\lambda(t)\}\subset\mathcal{H}^2(\widehat{A})+(\mathcal{H}^2(\widehat{A})^*)$. It follows that $\mathrm{span}\{h_R^{1/2}x\lambda(t): x\in M, \lambda(t), t\in \mathbb{Q}_D\} \subset \mathcal{H}^2(\widehat{A})+(\mathcal{H}^2(\widehat{A})^*)$. But by definition $R$ is the $\sigma$-weak closure of $\mathrm{span}\{x\lambda(t): x\in M, \lambda(t), t\in \mathbb{Q}_D\}$. So for any $g\in R$, we may select a net $\{g_\alpha\}$ in this span converging $\sigma$-weakly to $g$. Using the fact that $h_R^{1/2}\in L^2(R)$, it is now an exercise to see that then $\{h_R^{1/2}g_\alpha\}$ converges weakly to $h_R^{1/2}g$. Hence $h^{1/2}_RR$ is contained in the $L^2$-weak-closure of $\mathrm{span}\{h_R^{1/2}x\lambda(t): x\in M, \lambda(t), t\in \mathbb{Q}_D\}$. But since this is a convex set, the weak and norm closures agree. So the norm closure of this space must contain $h^{1/2}_RR$, which is known to be dense in $L^2(R)$. It follows that the norm-closed subspace $\mathcal{H}^2(\widehat{A})+(\mathcal{H}^2(\widehat{A})^*)$ of $L^2(R)$ contains a dense subspace of $L^2(R)$, and hence that $\mathcal{H}^2(\widehat{A})+(\mathcal{H}^2(\widehat{A})^*)=L^2(R)$, as claimed. 

The claim regarding $L^2(R_n)$ follows from the fact that the extension of $\Phi_n$ to $L^2(R)$, maps $L^2(R)$ onto $L^2(R_n)$, and $\mathcal{H}^2(\widehat{A})$ onto $\mathcal{H}^2(\widehat{A}_n).$
\end{proof}

\begin{lemma}\label{unsep}
 Let $A$ be an analytically conditioned algebra.  If any $f\in L^1(M)^+$ which is in the annihilator of $A_0$ belongs to $L^1(\mathcal{D})$, then also 
\begin{itemize}
\item any $f\in L^1(R)^+$ which is in the annihilator of $\widehat{A}_0$ belongs to $L^1(\widehat{\mathcal{D}})$,
\item and for any $n$, any $f\in L^1(R_n)^+$ which is in the annihilator of $(\widehat{A}_n)_0$, belongs to $L^1(\mathcal{D}_n)$.
\end{itemize}
\end{lemma}

\begin{proof} Let $tr_R$ be the canonical trace functional on $L^1(R)$. We remind the reader that the dual action of $L^1(R)$ on $R$, is given by $tr_R(ab)$ where $a\in L^1(R)$ and $b\in R$. As was noted in the proof of the previous Lemma, we may for any $n$ regard each of $L^1(R_n)$ and $L^1(M)$ as subspaces of $L^1(R)$. Suppose that $A$ satisfies the condition stated in the hypothesis, and let $f\in L^1(R)^+$ be given such that $f$ annihilates $\widehat{A}_0$. 

To prove the first claim, we need to show that then $f\in L^1(\widehat{\mathcal{D}})$. Now since $A_0\subset \widehat{A}_0$, we will for any $a\in A_0$ have that 
$$0=tr_R(fa)=tr_R(\Phi(fa))=tr_R(\Phi(f)a).$$Hence $\Phi(f)\in L^1(M)^+$ with $\Phi(f)\perp A_0$. It therefore follows from the hypothesis that $\Phi(f)\in L^1(\mathcal{D})$. 

Now notice that for any $t,s \in \mathbb{Q}_D$, it is trivially true that $\lambda(t)\widehat{A}_0\lambda(s)\subset \widehat{A}_0$. Using this fact, it is a simple exercise to show that each of $\lambda(t)^*f\lambda(t)$, $(\I+\lambda(t)^*)f(\I+\lambda(t))$, and $(\I-i\lambda(t)^*)f(\I+i\lambda(t))$ are also positive elements of $L^1(R)$ which are orthogonal to  $\widehat{A}_0$. Hence by the same argument as before, each of $\Phi(\lambda(t)^*f\lambda(t))$, $\Phi((\I+\lambda(t)^*)f(\I+\lambda(t)))= \Phi(f) + \Phi(\lambda(t)^*f) + \Phi(f\lambda(t))+ \Phi(\lambda(t)^*f\lambda(t))$, and $\Phi((\I-i\lambda(t)^*)f(\I+i\lambda(t)))=\Phi(f) -i\Phi(\lambda(t)^*f) + i\Phi(f\lambda(t))+ \Phi(\lambda(t)^*f\lambda(t))$, also belong to $L^1(\mathcal{D})$. Simple arithmetic now leads to the conclusion that $$\Phi(f\lambda(t))\in L^1(\mathcal{D})\mbox{ for each }t\in \mathbb{Q}_D.$$

We remind the reader that on elements of the form $\lambda(t)b$ where $t\in \mathbb{Q}_D$ and $b\in M$, the action of $\widehat{\mathcal{E}}$ and $\Phi$ and are respectively given by $\widehat{\mathcal{E}}(\lambda(t)b)= \lambda(t)\mathcal{E}(b)$ and $$\Phi(\lambda(t)b)=\left\{\begin{array}{ll} b &\quad \mbox{if }t=0\\ 0 & \quad \mbox{otherwise} \end{array}\right..$$It easily follows from this that $$\Phi(\widehat{\mathcal{E}}(\lambda(t)b))= \mathcal{E}(\Phi(\lambda(t)b)).$$ Since the span of elements of the form $\lambda(t)b$ is $\sigma$-weakly dense in $R$, the normality of each of $\widehat{\mathcal{E}}$ and $\Phi$, now leads to the conclusion that $\Phi\circ\widehat{\mathcal{E}}= \mathcal{E}\circ\Phi$. On combining this fact with the fact that $\Phi(f\lambda(t))\in L^1(\mathcal{D}))$ for each $t\in \mathbb{Q}_D$, it now follows that
\begin{eqnarray*} 
tr_R(f\lambda(t)b) &=& tr_R(\Phi(f\lambda(t)b))\\
&=& tr_R(\Phi(f\lambda(t))b)\\
&=& tr_R(\mathcal{E}\circ\Phi(f\lambda(t))b)\\
&=& tr_R(\Phi(\widehat{\mathcal{E}}(f\lambda(t)))b)\\
&=& tr_R(\Phi(\widehat{\mathcal{E}}(f)\lambda(t))b)\\
&=& tr_R(\Phi(\widehat{\mathcal{E}}(f)\lambda(t)b))\\
&=& tr_R(\widehat{\mathcal{E}}(f)\lambda(t)b)
\end{eqnarray*}
Once again the fact that $\mathrm{span}\{\lambda(t)))b: t\in \mathbb{Q}_D, b\in M\}$ is $\sigma$-weakly dense in $R$, now ensures that $tr_R(fg)= tr_R(\widehat{\mathcal{E}}(f)g)$ for any $g\in R$. Hence $f=\widehat{\mathcal{E}}(f)$ as required.

The second claim now easily follows from the first. To see this let $f\in L^1(R_n)^+$ be given with $f\perp (\widehat{A}_n)_0$. We need to show that then $f\in L^1(\widehat{\mathcal{D}}_n)=L^1(\widehat{\mathcal{D}})\cap L^1(R_n)$. Using the fact that $\Phi_n((\widehat{A})_0)=(\widehat{A}_n)_0$, it now easily follows that $$tr_R(fa) = tr_R(\Phi_n(fa)) = tr_R(f\Phi_n(a)) =0$$for any $a\in \widehat{A}_0$. Hence by the first part $f\in L^1(\widehat{\mathcal{D}})$ as required.
\end{proof}

We are now ready to prove our main theorem.

\begin{theorem}  \label{Co}
 Let $A$ be an analytically conditioned algebra.  Then the following are equivalent:
\begin{itemize} \item [(i)]  $A$ is maximal subdiagonal,
  \item [(ii)]  For every right
$A$-invariant subspace $X$ of $L^2(M)$, the right wandering subspace
 $W$ of $X$ satisfies
$W^* W \subset L^1({\mathcal D})$, and $W^* (X \ominus [W A]_2) = (0)$.
\item[(iii)] $L^2(M)=\mathcal{H}^2(A)\oplus(\mathcal{H}_0^2(A))^*$, and any $f\in L^1(M)^+$ which is in the annihilator of $A_0$ belongs to $L^1(\mathcal{D})$.
\end{itemize}
\end{theorem}

\begin{proof}  The fact that (i) implies (ii) is
proved in Theorem \ref{inv1}.  We proceed to prove that (ii) implies (iii). 
To this end, let $g \in L^1_+(M)$ be given with $\tau(g A) = 0$. 
Let $f = |g|^{\frac{1}{2}}$. Clearly $f\in L^2(M)$, 
and $f^2=g$. Now set $X = [f A]_2$.
Note that $f \perp [f A_0]_2$ since if $a_n
\in A_0$ with $f a_n \to k$ in $L^2$-norm, then $tr(f^* k) =
\lim_n tr(f^* f a_n) = \lim_n tr(g a_n)=0$. In particular, the fact that $f \perp
[f A_0]_2=[XA_0]_2$, ensures that $f \in X \ominus [XA_0]_2 = W$. So by
hypothesis, $f^2 = g \in L^1({\mathcal D})$.

Next set $X = L^2(M)\ominus (\mathcal{H}^2_0(A))^*$. We will deduce that $A$ satisfies
$L^2$-density. That is that $X = \mathcal{H}^2(A)$.
To this end, note that $X$ is right $A$-invariant. To see this first note that 
since $A$ is subdiagonal, $\{h^{1/2}a_0^*: a_0\in A_0\}$ 
is dense in $(\mathcal{H}^2_0(A))^*$. So $f \in L^2(M)$ is 
orthogonal to $(\mathcal{H}^2_0(A))^*$ if and only if 
$tr(a_0h^{1/2}f)=tr((h^{1/2}a_0^*)^*f)=0$ for every $a_0\in A_0$. Given $f\in X$, $a\in A$
and $a_0\in A_0$, the fact that then $aa_0\in A_0$, ensures that we will then 
have that $tr(a_0h^{1/2}(fa))=tr(aa_0h^{1/2}f)=0$ for every $a_0\in A_0$. 
Hence $fa\in L^2(M)\ominus (\mathcal{H}^2_0(A))^*=X$ as required. 
 
It is easy to see that $h^{1/2} \in X$ where $h=\frac{d\widetilde{\varphi}}{d\tau_L}$. (This is an immediate consequence of the fact that 
$\{a_0h^{1/2}: a_0\in A_0\}$ is dense in $\mathcal{H}^2_0(A)$, and that 
$tr(h^{1/2}(ah^{1/2}))=\varphi(a)=0$ for all $a\in A_0$.) In fact  $h^{1/2} \in W=X\ominus[XA_0]_2$ since for any $a_0\in A_0$ and $f\in X$
we already know that $0=tr(a_0h^{1/2}f)=tr(h^{1/2}(fa_0))$. This forces $h^{1/2}(X \ominus [W A]_2)\subset W^*(X \ominus [W A]_2)=(0)$. 
The injectivity of the embedding $L^2(M)\to L^1(M): s\to h^{1/2}s$ now ensures that $X \ominus [W A]_2=(0)$. However the fact that 
$h^{1/2} \in W$ also ensures that $h^{1/2}W \subset W^* W  \subset L^1({\mathcal D})$. For any $w\in W$ we will then have that 
$h^{1/2}w=\mathcal{E}_1(h^{1/2}w)=h^{1/2}\mathcal{E}_2(w)$. On once again appealing to the injectivity of the embedding $L^2(M)\to L^1(M): s\to h^{1/2}s$, we may now conclude that $w=\mathcal{E}_2(w)\in L^2(\mathcal{D})$ for any $w\in W$. So $X =
[WA]_2 \subset [L^2(\mathcal{D})A]_2\subset \mathcal{H}^2(A)$. 
The converse inclusion $\mathcal{H}^2(A) \subset X$
follows from the fact that $\mathcal{H}^2(A)$ is orthogonal to $(\mathcal{H}^2_0(A))^*$.

We prove that (iii)$\Rightarrow$(i). Given that (iii) holds, it then follows from Lemmata \ref{L2} and \ref{unsep} that (iii) also holds when the pair $(M,A)$ is replaced by any of the pairs $(R_n, \widehat{A}_n)$. But each $R_n$ is a finite von Neumann algebra, and the stated property does not just hold in terms of $(R_n, \widehat{A}_n, \widehat{\varphi}_n, \widehat{\mathcal{E}})$, but also in terms of $(R_n, \widehat{A}_n, \tau_n, \widehat{\mathcal{E}})$ where $\tau_n$ is the canonical finite trace on $R_n$. This bears some justification, and hence we pause to substantiate this claim. Firstly note that the canonical trace on $R_n$ is of the form $\tau_n(\cdot) = \widehat{\varphi}_n(e^{-a_n}\cdot)$ for some element $a_n$ in the von Neumann algebra generated by the operators $\{\lambda(t):t\in \mathbb{Q}_D\}\subset R$. Hence the fact that $\widehat{\varphi}_n\circ\widehat{\mathcal{E}}$ ensures that also $\tau_n(\widehat{\mathcal{E}}(\cdot))=\widehat{\varphi}_n(e^{-a_n}\widehat{\mathcal{E}}(\cdot)) = \widehat{\varphi}_n(\widehat{\mathcal{E}}(e^{-a_n}\cdot))=\widehat{\varphi}_n(e^{-a_n}\cdot)=\tau_n$. So $\widehat{A}_n$ is indeed also a \emph{tracial} subalgebra of $R_n$. It further follows from Corollary II.38 of \cite{Tp} that there exists a topological isomorphism from the $\tau$-measurable operators affiliated with $R_n\rtimes_{\widehat{\varphi}_n} \mathbb{R}$, to those affiliated with $R_n\rtimes_{\tau_n} \mathbb{R}$, in a manner which identifies the $L^p$ spaces corresponding to the two contexts. The Remark immediately following \cite[Corollary II.38]{Tp} moreover informs us that the Haagerup $L^p$-spaces corresponding to the context $R_n\rtimes_{\tau_n} \mathbb{R}$, are of the form $\{f\otimes\exp(\cdot/p):f\in L^p(R_n, \tau_n)\}$, where $L^p(R_n, \tau_n)$ are the ``tracial'' $L^p$-spaces. If one carefully follows the action of these maps, it can be seen that in the case of $R_n$, (iii) holds for the ``Haagerup'' context, if and only if it holds for the ``tracial'' context. 

For the case of finite von Neumann algebras it is known that condition (iii) is equivalent to the condition that $\widehat{A}_n^*+\widehat{A}_n$ is $\sigma$-weakly dense in $R_n$ (\cite{BL1}, \cite{BLsurvey}). Hence for each $n\in \mathbb{N}$, we have that $\widehat{A}_n^*+\widehat{A}_n$ is $\sigma$-weakly dense in $R_n$. Thus the $\sigma$-weak closure of $\cup_{n\in \mathbb{N}}(\widehat{A}_n^*+\widehat{A}_n)$ includes $\cup_{n\in \mathbb{N}}R_n$. But $\cup_{n\in \mathbb{N}}R_n$ is $\sigma$-weakly dense in $R$. Hence the same must be true of $\cup_{n\in \mathbb{N}}(\widehat{A}_n^*+\widehat{A}_n)$. But $\cup_{n\in \mathbb{N}}(\widehat{A}_n^*+\widehat{A}_n)\subset \widehat{A}^*+\widehat{A}$. So $\widehat{A}^*+\widehat{A}$ is $\sigma$-weakly dense in $R$. By the $\sigma$-weak continuity of $\Phi$, $\Phi(\widehat{A}^*+\widehat{A})= A^*+A$ is then $\sigma$-weakly dense in $\Phi(R)=M$. Hence (i) holds.
\end{proof}


\begin{thebibliography}{8888} 
\bibitem{PTAG}   W. B. Arveson, {\em Prediction theory and 
group representations,}  Ph. D. Thesis, UCLA, 1964.

 \bibitem{AIOA}  W. B. Arveson, Analyticity
in operator algebras, {\em Amer.\ J.\ Math.\ } {\bf 89} (1967),
578--642.

\bibitem{BX}  T. Bekjan and Q. Xu, Riesz and Szeg\"o type factorizations
for noncommutative Hardy spaces,  J Operator Theory 62 (2009), 215–231. 

\bibitem{Bek}  T Bekjan, Riesz factorization of Haagerup noncommutative Hardy spaces, preprint. 

\bibitem{Bek2}  T Bekjan, Noncommutative symmetric Hardy spaces, \textit{Integr Equ Oper Theory} \textbf{81} (2015), 191-212. 

\bibitem{BL1}    D. P. Blecher and L. E. Labuschagne, Characterizations of noncommutative $H^\infty$,   {\em
 Integr.\ Equ.\ Oper.\ Theory}  {\bf 56} (2006), 301-321.

\bibitem{BL-FMRiesz} DP Blecher and LE Labuschagne, Noncommutative function theory and 
unique extensions, \textit{Studia Mathematica} \textbf{178}(2007), 177-195.

\bibitem{BLsurvey} DP Blecher and LE Labuschagne, Von Neumann algebraic $H^p$ theory. 89-114. In: K Jarosz (editor),
\textit{Proceedings of the Fifth Conference on Function Spaces}, \textit{Contemporary Mathematics} \textbf{435}, American Mathematical Society, 2007.

\bibitem{BL-outer} DP Blecher and LE Labuschagne, Applications of the Fuglede-Kadison determinant: Szeg\"o's theorem and outers for noncommutative $H^p$, \textit{Transactions of the AMS} \textbf{360}(2008), 6131-6147.

\bibitem{BL-Beurling} DP Blecher and LE Labuschagne, A Beurling Theorem
for noncommutative $L^p$, \textit{Journal of Operator Theory}, \textbf{59}(2008), 29-51.

\bibitem{BL-outer2} DP Blecher and LE Labuschagne, Outers for noncommutative $H^p$ revisited, \textit{Studia Math.} 217(3) (2013), 265-287.

\bibitem{CHS} Y Chen, D Hadwin, J Shen, A non-commutative Beurling's theorem with respect to unitarily invariant norms, preprint. (See 	arXiv:1505.03952 [math.OA])

\bibitem{FKa}  B. Fuglede and R. V. Kadison,  Determinant theory in finite factors, {\em Ann.\ of Math.} {\bf  55}  (1952),
 520-530.

\bibitem{Gol} S Goldstein, Conditional expectation and stochastic integrals in non-commutative $L^p$ spaces, \textit{Math Proc Camb Phil Soc} \textbf{110}(1991), 365--383

\bibitem{HJX} U Haagerup, M Junge, and Q Xu, A reduction method for noncommutative Lp-spaces and applications, \textit{Trans Amer Math
Soc} \textbf{362} (2010), 2125–-2165.

\bibitem{HL}  H. Helson and D. Lowdenslager, Prediction theory and Fourier series
in several variables, {\em Acta Math.} {\bf 99} (1958), 165-202.

\bibitem{Ho}  K. Hoffman, Analytic functions and logmodular
Banach algebras, {\em  Acta Math.} {\bf 108} (1962), 271-317.

\bibitem{jig2} G-X Ji, A noncommutative version of $H^p$and characterizations of subdiagonal subalgebras, \textit{Integr Equ Oper Theory}\textbf{72} (2012), 131--149.

\bibitem{jig1} G-X Ji, Analytic Toeplitz algebras and the Hilbert transform associated with a subdiagonal algebra, \textit{Sci China Math}\textbf{57} (3) (2014), 579--588.

\bibitem{JOS} G-X Ji, T Ohwada and K-S Saito, Certain structure of subdiagonal algebras, \textit{J Operator Theory} \textbf{39}(1998), 309--317.

\bibitem{jisa}  G-X Ji and K-S Saito, Factorization in Subdiagonal Algebras, \textit{J Funct Anal}\textbf{159} (1998), 191--201.

\bibitem{Junge} M Junge, Doob's inequality for noncommutative martingales, \textit{J Reine Angew Math} \textbf{549}(2002), 149-190.

\bibitem{js} M Junge and D Sherman, Noncommutative $L^p$-modules, \textit{Journal of Operator Theory}, \textbf{53}(1)(2005), 3-34.

\bibitem{Kai}  S. Kaijser,
 On Banach modules II. Pseudodeterminants and traces, {\em
  Math. Proc. Cambridge Philos. Soc.} \ {\bf 121}  (1997),  325--341.

\bibitem{Kos} H Kosaki, Applications of the complex interpolation method to a von Neumann algebra: Noncommutative $L^p$-spaces, \textit{ J Funct Anal} \textbf{56}(1984), 29--78.

\bibitem{L-Szego}  L E Labuschagne,  A noncommutative Szeg\"o theorem
for subdiagonal subalgebras of von Neumann algebras, {\em Proc.
Amer. Math. Soc.}, {\bf 133} (2005), 3643-3646.

\bibitem{pru} B Prunaru, Toeplitz and Hankel operators associated with subdiagonal algebras, \textit{Proc. Amer. Math. Soc.} \textbf{139} (2010), 1387--1396.

\bibitem{sager} L Sager, A Beurling-Blecher-Labuschagne theorem for noncommutative Hardy spaces associated with semifinite von Neumann algebras, preprint. (See arXiv:1603.01735 [math.OA])

\bibitem{SW}  T. P.  Srinivasan and J-K. Wang, {\em Weak*-Dirichlet algebras,}
 In {\em Function algebras}, Ed. Frank T. Birtel, Scott Foresman and Co., 1966, 216-249.
 
\bibitem{Tak} M Takesaki, \textit{Theory of Operator Algebras, Vol I,II,III}, Springer, New York, 2003.
 
\bibitem{Tp} M Terp, $L^p$ \textit{spaces associated with von Neumann algebras}, Notes, K{\o}benhavns Universitet, Matematisk Institut, Rapport No. 3a/3b, K{\o}benhavn, 1981.

\bibitem{Ueda} Y Ueda, On peak phenomena for non-commutative $H^\infty$, \textit{Math Ann} \textbf{343} (2009), 421–-429. 

\bibitem{Xu}  Q Xu,  On the maximality of subdiagonal algebras, {\em J. Operator Theory}   {\bf 54}  (2005), 137--146.
\end{thebibliography}
\end{document}